\documentclass[11pt]{article}
\usepackage[a4paper,margin=1in]{geometry}
\usepackage[numbers]{natbib}
\usepackage{amsmath,amssymb,amsthm,amsfonts}
\usepackage{hyperref}
\hypersetup{hidelinks}
\usepackage{booktabs}
\usepackage{enumitem}
\newtheorem{theorem}{Theorem}[section]
\newtheorem{definition}[theorem]{Definition}
\newtheorem{lemma}[theorem]{Lemma}

\newtheorem{assumption}[theorem]{Assumption}

\newtheorem{proposition}[theorem]{Proposition}

\newcommand{\R}{\mathbb{R}}
\newcommand{\Sym}{\mathbb{S}}
\newcommand{\trace}{\mathrm{tr}}
\newcommand{\diag}{\mathrm{diag}}
\newcommand{\Diag}{\mathrm{Diag}}
\newcommand{\st}{\mathrm{s.t.}}
\newcommand{\rank}{\mathrm{rank}}

\title{Exact low-dimensional reformulations for regularized spectral approximation}

\author{
Shengxiang Deng\thanks{School of Data Science, Fudan University, Shanghai, China.}
\and
Xudong Li\thanks{School of Data Science, Fudan University, Shanghai, China. Email: \texttt{lixudong@fudan.edu.cn}.}
\and
Yangjing Zhang\thanks{State Key Laboratory of Mathematical Sciences, Academy of Mathematics and Systems Science, Chinese Academy of Sciences, Beijing, China. Email: \texttt{yangjing.zhang@amss.ac.cn}.}
}
\date{August 27, 2026}

\begin{document}
\maketitle

\begin{abstract}
We study exact low-dimensional reformulations for a regularized spectral approximation problem and its weighted variant. In the unweighted setting, weak-majorization monotonicity of the fidelity term enables an exact reformulation over singular values. In the weighted setting, we establish such a reformulation under compatibility and simultaneous diagonalization conditions. We further show, via a counterexample, that this exact reformulation can fail in the absence of simultaneous diagonalization.
\end{abstract}


\section{Introduction}

	Regularized spectral approximation problems arise when one seeks a matrix $X$ that balances data fidelity and spectral regularization, possibly under additional spectral constraints. Typical models take the form of a data-fitting term involving $X-A$ or $A-BXC$ with $A$ being a given data matrix and $B$, $C$ being given weight matrices, together with a spectral regularizer on $X$ and, in some cases, a spectral constraint. Such formulations appear throughout matrix approximation \cite{eckart1936approximation,gavish2014optimal}, inverse problems and low-complexity modeling \cite{fan2019factor,deng2025alternating}, where a critical step is whether the original matrix optimization problem can be reduced exactly to a lower-dimensional problem corresponding to eigenvalues or singular values.

	A natural starting point is the Eckart--Young--Mirsky theorem \cite{eckart1936approximation,mirsky1960symmetric}, which gives a complete singular-value reformulation of the best low-rank approximation problem under unitarily invariant norms. However, this classical reduction is tied to a highly special setting. In many modern regularized approximation models, the loss term is no longer a unitarily invariant norm, the loss itself and the regularizer may be nonconvex, and weighted forms such as $A-BXC$ destroy the underlying unitary invariance. In these situations, it is no longer clear whether an exact eigenvalue or singular-value reformulation survives, or whether optimal solutions retain any explicit structure.

	This paper studies exact low-dimensional reformulations for a regularized spectral approximation problem and its weighted variant. Rather than focusing on algorithm design, we aim to identify structural conditions under which the matrix problem admits an equivalent vector formulation. In this sense, our goal is to clarify when regularized spectral approximation remains exactly reducible despite the presence of general spectral losses, regularizers, and weight matrices.

We start with the regularized spectral approximation problem
\begin{equation} \label{eq:intro_prob1}
	\min_{X \in \R^{m \times n}} F_1(X-A) + \mu F_2(X) \quad \st \quad G(X) \leq c,
\end{equation}
where $A \in \R^{m \times n}$ and $c\in \R$ are given, $F_1$, $F_2$, and $G$ are spectral functions, and $\mu \ge 0$ is a regularization parameter. Here $F_1(X-A)$ serves as the fidelity term and $F_2(X)$ as the regularizer. Model \eqref{eq:intro_prob1} subsumes a variety of problems arising in multivariate regression, subspace learning, and matrix recovery; see, e.g., \cite{yuan2007dimension,liu2013lrr,lu2014generalized,fan2019factor,bertsimas2023sparse}.

Existing exact reformulation results \cite{mirsky1960symmetric,yu2012rank,deng2025alternating} for \eqref{eq:intro_prob1} typically require both $F_1$ and $F_2$ to be unitarily invariant norms and restrict $G$ to the rank function, i.e., $G(X)=\rank(X)$. In contrast, we require only that $F_1$, $F_2$, and $G$ are spectral functions. We further assume that the absolutely symmetric function $f_1$ associated with the fidelity term $F_1$ is monotone under weak majorization; see Assumption~\ref{ass:monotonicity}. This condition is satisfied by all unitarily invariant norms, their non-decreasing compositions, and certain nonconvex spectral functions. Under this assumption, we establish an exact low-dimensional reduction of \eqref{eq:intro_prob1}; see \eqref{prob:vector}.

	We next consider the weighted regularized spectral approximation problem
\begin{equation*}
	\min_{X \in \R^{m \times n}} \ell(\|A - BXC\|_{\rm UI}) + \mu F(X) \quad \st \quad G(X) \leq c,
\end{equation*}
where $A$, $B$, and $C$ are given matrices of compatible dimensions, $\|\cdot\|_{\rm UI}$ denotes a unitarily invariant norm, $F$ and $G$ are spectral functions, $\ell:\R_+\to(-\infty,\infty]$ is proper and non-decreasing, and $\mu\ge0$. The weight matrices $B$ and $C$ destroy the unitary invariance present in the unweighted setting, making the problem substantially more difficult.

In the rank-constrained Frobenius-loss setting, exact reduction results were established by Sondermann \cite{sondermann1986best} and later rediscovered by Friedland and Torokhti \cite{friedland2007generalized}. Related one-sided and fixed-subspace variants were studied in \cite{golub1987generalization}. Beyond the Frobenius setting, Yu and Schuurmans \cite{yu2012rank} studied weighted problems involving general unitarily invariant norms under simultaneous diagonalization assumptions. We extend this line of work from rank-constrained norm-based models to a broader regularized spectral framework. Under Assumption~\ref{ass:weighted_setup}, especially conditions (iii)--(v), we prove that whenever the weighted problem admits an optimal solution, it also admits an optimal solution with a diagonal core in the singular-vector coordinates, and hence an exact low-dimensional reformulation; see \eqref{prob2:vector}.

	\section{Preliminaries and Definitions}

	In this section, we introduce the necessary notation and definitions that will be used. Let $\R^{m \times n}$ be the space of $m \times n$ real matrices.  We write $\R^m$ for the $m$-dimensional Euclidean space and $\R^m_{++}$ for its strictly positive orthant. We denote by $\Sym^m$ the space of real symmetric $m\times m$ matrices and by $\Sym_{++}^m$ the cone of symmetric positive definite matrices.
	We use $\lambda(X) \in \R^m$ to denote the vector of eigenvalues of $X\in\Sym^m$, arranged in non-increasing order.
	Let $q = \min\{m, n\}$.
	The vector of singular values of $X\in\R^{m \times n}$ is denoted by $\sigma(X)\in \R^q$ with the components being arranged in non-increasing order $\sigma_1(X) \geq \dots \geq \sigma_q(X) \geq 0$. For $x \in \R^q$, we denote by $x_{[1]} \ge \dots \ge x_{[q]}$ the components of $x$ arranged in non-increasing order.
	A proper function is a function that never takes the value $-\infty$ and is not identically $+\infty$. The characteristic function $\mathbb{I}(\mathcal{E})$ equals $1$ when the event $\mathcal{E}$ is true and $0$ otherwise. For a vector $x$, $\Diag(x)$ denotes the generalized diagonal matrix with diagonal vector $x$; its dimensions are determined by context and need not be square. For a matrix $X\in \R^{m \times n}$, $\diag(X)\in\R^q$ denotes the vector of diagonal entries of $X$. We write $\trace(X)$ for the trace of a matrix. For matrices, $\|\cdot\|_2$, $\|\cdot\|_F$, and $\|\cdot\|_{S_p}$ denote the spectral, Frobenius, and Schatten-$p$ (quasi-)norms, respectively. For a matrix $U$ with orthonormal columns, we use $U^\perp$ to denote its orthogonal complement, so that $[\,U\ \ U^\perp\,]$ is orthogonal.

	We first recall the definitions of spectral functions and unitarily invariant norms.

	\begin{definition}[Symmetric and Absolutely Symmetric Function]
		A proper function $f : \R^q \to (-\infty, \infty]$ is said to be symmetric if $f(Px) = f(x)$ for every permutation matrix $P$, and absolutely symmetric if $f(PJx) = f(x)$ for every permutation matrix $P$ and every diagonal sign matrix $J$ with $J_{ii}\in\{\pm1\}$.
	\end{definition}

	\begin{definition}[Spectral Function \cite{lewis1996derivatives}]
		A proper function $F$ is called a spectral function if either there exists an absolutely symmetric function $f : \R^q \to (-\infty, \infty]$ such that
		\[
		F(X) = f(\sigma(X)), \quad \forall X \in \R^{m \times n},
		\]
		or there exists a symmetric function $f : \R^m \to (-\infty, \infty]$ such that
		\[ F(X) = f(\lambda(X)), \quad \forall X \in \Sym^m. \]
	\end{definition}

	Spectral functions are inherently unitarily invariant, since they depend only on singular values (or eigenvalues in the symmetric case). Table~\ref{tab:spectral_functions} lists several standard examples and their associated absolutely symmetric vector functions. In this paper, we focus on spectral functions defined through singular values. Analogous results can also be established for eigenvalue-based spectral functions.

	\begin{table}[h]
		\centering
		\caption{Correspondence between absolutely symmetric functions and spectral functions. ${\rm dom}(f)=\R^{q}$, ${\rm dom}(F)=\R^{m\times n}$.
		}
		\label{tab:spectral_functions}
		\renewcommand{\arraystretch}{1.5}
		\begin{tabular}{lc}
			\toprule
			$f(x)$ &   $F(X)=f(\sigma(X))$\\
			\midrule
			$\sum_{i=1}^{q} |x_{i}|$  & $\|X\|_{*}$  \\
			$\max_{i} |x_{i}|$  & $\|X\|_{2}$  \\
			$(\sum_{i=1}^{q} |x_{i}|^2)^{1/2}$  & $\|X\|_{F}$  \\
			$\sum_{i=1}^{k} |x|_{[i]}$  & $\|X\|_{(k)}$  \\
			$(\sum_{i=1}^{q} |x_{i}|^p)^{1/p}, \, p > 0$  & $\|X\|_{S_p}$  \\
			$\sum_{i=1}^{q} \mathbb{I}(x_i \neq 0)$  & $\rank(X)$  \\
			\bottomrule
		\end{tabular}
	\end{table}

	Unitarily invariant norms are particularly important for the weighted model, so we also recall the corresponding definitions.

	\begin{definition}[Unitarily Invariant Norm]
		A matrix norm $\|\cdot\|$ on $\R^{m \times n}$ is called a unitarily invariant norm if
		\[
		\|U A V^T\| = \|A\|
		\]
		for all $A \in \R^{m \times n}$ and all orthogonal matrices $U \in \R^{m \times m}, V \in \R^{n \times n}$.
	\end{definition}

	Common examples of unitarily invariant norms include the spectral norm, the Frobenius norm, and the nuclear norm. The following theorem records the standard correspondence between unitarily invariant norms and symmetric gauge functions\footnote{A function $\phi:\R^q\to\R_+$ is called a symmetric gauge function if it is absolutely symmetric and $\phi$ is a norm on $\R^q$.} of singular values.
	\begin{theorem}[{\cite[Theorem~3.5.18]{horn1994topics}}]\label{thm:uis}
			If $\|\cdot\|$ is a unitarily invariant norm on $\R^{m \times n}$, then there exists a symmetric gauge function $\phi$ on $\R^q$ such that $\|X\| = \phi(\sigma(X))$ for every $X \in \R^{m \times n}$. Conversely, if $\phi$ is a symmetric gauge function on $\R^q$, then $\|X\| = \phi(\sigma(X))$ is  a unitarily invariant norm on $\R^{m \times n}$.
	\end{theorem}

	We next turn to majorization, which provides the ordering used to formulate the monotonicity assumption in Section \ref{sec:main_results}.
	\begin{definition}[Majorization]
		For any vectors $x, y \in \R^q$, let $x_{[1]} \ge \dots \ge x_{[q]}$ and $y_{[1]} \ge \dots \ge y_{[q]}$ denote their components arranged in non-increasing order. We say that $x$ is majorized by $y$, denoted as $x \prec y$, if
		\[
		\sum_{i=1}^k x_{[i]} \le \sum_{i=1}^k y_{[i]} \quad \forall k = 1, \dots, q-1, \,\, \text{and} \,\, \sum_{i=1}^q x_{[i]} = \sum_{i=1}^q y_{[i]}.
		\]
		Furthermore, if the equality condition at $k=q$ is replaced by inequality ($\le$), we say that $x$ is weakly majorized by $y$, denoted as $x \prec_w y$.
	\end{definition}

	The following standard inequalities connect diagonal entries, singular values, and perturbations. The proofs can be found  in \cite{fan1951maximum,mirsky1960symmetric},
	\cite[Proposition~B.6]{marshall2011inequalities}, and \cite[Theorem~3.4.5]{horn1994topics}.

	\begin{theorem}\label{thm:diag}
	For any $A,X \in \R^{m \times n}$, it holds that
	\begin{enumerate}[label=(\roman*),leftmargin=*]
		\item $|\diag(A)| \prec_w \sigma(A)$;

		\item $|\sigma(X)-\sigma(A)| \prec_w \sigma(X-A)$.
	\end{enumerate}

	\end{theorem}

	The following notion of monotonicity under weak majorization will be used in the subsequent analysis.
	\begin{definition}[Monotonicity Under Weak Majorization]\label{def:weak_majorization_monotonicity}
	A function $f:\R^q\to(-\infty,\infty]$ is said to be monotone under weak majorization on $\R^q_+$ if, for any $x,y\in\R^q_+$,
	\[
	x\prec_w y \;\Rightarrow\; f(x)\le f(y).
	\]
	\end{definition}

	A standard result \cite[Theorem~A.8]{marshall2011inequalities} states that monotonicity under weak majorization on $\R^q_+$ is equivalent to being componentwise non-decreasing and Schur-convex on $\R^q_+$. Here, componentwise non-decreasing means that if $0\le x_i \le y_i$ for all $i=1,\dots,q$, then $f(x)\le f(y)$. This characterization shows that all unitarily invariant norms are monotone under weak majorization.
	We shall also emphasize that monotonicity under weak majorization does not require convexity. For example, the nonconvex log-type spectral function
	\[
	F(X) = \log(1 + \|X\|_*/\gamma),
	\qquad \gamma > 0,
	\]
	has the associated absolutely symmetric function
	$
	f(x) = \log(1 + \|x\|_1/\gamma),
	$
	which is monotone under weak majorization on $\R^q_+$.

	\section{Main Results}\label{sec:main_results}

		\subsection{Regularized Spectral Approximation}

	In this subsection, we derive the exact low-dimensional reformulation for the regularized spectral approximation problem. We recall its formulation here for convenience:
	\begin{equation} \label{prob:general}
		\min_{X \in \R^{m \times n}} F_1(X-A) + \mu F_2(X) \ \ \st \ \ G(X) \leq c.
	\end{equation}
	We make the following assumptions on \eqref{prob:general}.

	\begin{assumption} \label{ass:monotonicity}
		For problem \eqref{prob:general}, we assume  that
		\begin{enumerate}[label=(\roman*),leftmargin=*]
			\item $F_1, F_2, G$ are spectral functions on $\R^{m \times n}$, associated with absolutely symmetric functions $f_1, f_2, g : \R^q \to (-\infty,\infty]$, respectively;
			\item $f_1$ is monotone under weak majorization on $\R^q_+$ in the sense of Definition~\ref{def:weak_majorization_monotonicity}.
		\end{enumerate}
	\end{assumption}

		Assumption~\ref{ass:monotonicity} imposes a weak-majorization monotonicity requirement only on the data-fitting term $F_1$, but not on $F_2$ or $G$. This weaker requirement keeps the framework compatible with a broad range of regularizers and constraints, such as the rank function (nonconvex and discontinuous) and other nonconvex surrogates, e.g., the Schatten-$p$ quasi-norm with $0 < p < 1$.

	\begin{theorem} \label{thm:general_spectral}
		Suppose Assumption~\ref{ass:monotonicity} holds. Let $A = U \Sigma V^T$ be the singular value decomposition (SVD) of $A$, where $\Sigma = \Diag(\sigma(A))$.
		Then problem~\eqref{prob:general} admits an optimal solution if and only if the reduced problem
		\begin{equation} \label{prob:vector}
			\begin{aligned}
					\min_{x \in \R^q} \quad & f_1(x - \sigma(A)) + \mu f_2(x) \\
					\st \ \quad & g(x) \le c, \\
					& x_1 \ge \cdots \ge x_q \ge 0
			\end{aligned}
		\end{equation}
		admits an optimal solution.
		Furthermore, if $x^*$ is an optimal solution to~\eqref{prob:vector},
		then any matrix of the form
		\[
		X^* = U \Diag(x^*) V^T
		\]
		is an optimal solution to~\eqref{prob:general}. Conversely, if $X^*$ is an optimal solution to~\eqref{prob:general}, then $\sigma(X^*)$ is an optimal solution to~\eqref{prob:vector}.
	\end{theorem}

	\begin{proof}
	We first establish a lower bound for the objective value in \eqref{prob:general}.
	Let $X \in \R^{m \times n}$ be an arbitrary feasible point of
	\eqref{prob:general}. Then $\sigma(X)$
	is feasible for~\eqref{prob:vector}.
	By Theorem~\ref{thm:diag}(ii),
	\[
	|\sigma(X) - \sigma(A)| \prec_w \sigma(X-A).
	\]
	Under Assumption~\ref{ass:monotonicity}, we obtain
	\begin{align*}
		& F_1(X-A) = f_1(\sigma(X-A)) \\
		\ge\ & f_1(|\sigma(X)-\sigma(A)|)
		= f_1(\sigma(X)-\sigma(A)).
	\end{align*}

	Consequently,  the objective value in \eqref{prob:general} is lower bounded by that in \eqref{prob:vector}
	\begin{align} \label{eq:fundamental_inequality}
		& F_1(X-A) + \mu F_2(X)\\
		\ge \ &
		f_1(\sigma(X)-\sigma(A))
		+ \mu f_2(\sigma(X)). \nonumber
	\end{align}

	Now, suppose that $x^*$ is an optimal solution to~\eqref{prob:vector}.
	Define $X^*= U \Diag(x^*) V^T$, which is obviously feasible for  \eqref{prob:general}.
	Note that $X^*-A = U\Diag(x^*-\sigma(A) )V^T$. Then we have
	\begin{align}\label{pf:in2}
		F_1(X^*-A) = f_1(\sigma(X^*-A)) = f_1(x^*-\sigma(A)).
	\end{align}
	Next we show that $X^*$ is optimal for  \eqref{prob:general}. Let $X \in \R^{m \times n}$ be an arbitrary feasible point of \eqref{prob:general}. Then $\sigma(X)$
	is feasible for~\eqref{prob:vector}. We have
	\begin{align*}
		&  F_1(X^*-A)  + \mu F_2(X^*) \\
		= \ &  f_1(x^*-\sigma(A))+ \mu f_2(x^*)  \\
		\le \ &  f_1(\sigma(X)-\sigma(A)) + \mu f_2(\sigma(X))  \\
		\le \ & F_1(X-A)  + \mu F_2(X), \quad
	\end{align*}
	where the first equality is by \eqref{pf:in2}, the second inequality is by the optimality of  $x^*$, and the last inequality is by \eqref{eq:fundamental_inequality}.
	Thus $X^*$ is optimal for~\eqref{prob:general}.

	Conversely, assume that $X^*$ is an optimal solution of \eqref{prob:general}. Define
	$x^*=\sigma(X^*)$, which is obviously feasible for \eqref{prob:vector}. Next we show that $x^*$ is optimal for \eqref{prob:vector}. Let $ x \in \R^q$ be an arbitrary
	feasible point of \eqref{prob:vector}. Then $X = U \Diag(x) V^T$ is feasible for \eqref{prob:general}. We have
	\begin{align*}
		& f_1(x^*-\sigma(A)) + \mu f_2(x^*) \\
		\le \ &    F_1(X^*-A)  + \mu F_2(X^*)  \\
		\le \ &   F_1(X-A)  + \mu F_2(X)  \\
		= \ &  f_1(x-\sigma(A)) + \mu f_2(x),
	\end{align*}
	where the first inequality is by \eqref{eq:fundamental_inequality}, the second inequality is by the optimality of  $X^*$, and the last equality is by the same argument as in \eqref{pf:in2}. Thus $x^*$  is optimal for~\eqref{prob:vector}.
	\end{proof}

As an illustration, consider the square-root low-rank matrix approximation problem
\begin{equation}\label{prob:low_rank_approx}
	\min \left\{ \|X-A\|_F + \mu F_2(X) \mid \rank(X)\leq k\right\},
\end{equation}
where $0<k\leq q=\min\{m,n\}$ is a given integer. Here, the regularizer $F_2$ can be the nuclear norm or a nonconvex Schatten-$p$ quasi-norm with $0<p<1$; see \cite{deng2025alternating,fan2019factor}.
Let $A=U\Diag(\sigma(A))V^T$ be the SVD of $A$, and let $f_2$ be the absolutely symmetric function associated with $F_2$. Then Theorem~\ref{thm:general_spectral} yields the following exact low-dimensional reformulation of \eqref{prob:low_rank_approx}:
\begin{equation*}
	\min \left\{ \|x-\sigma(A)\|_2 + \mu f_2(x) \mid \|x\|_0\le k,\ x_1\ge \cdots \ge x_q\ge 0 \right\}.
\end{equation*}
Here $\|\cdot\|_2$ denotes the Euclidean norm on vectors, and $\|\cdot\|_0$ denotes the number of nonzero entries.
For any feasible $x$, since $\|x\|_0\le k$ and $x_1\ge \cdots \ge x_q\ge 0$, we necessarily have $x_{k+1}=\cdots=x_q=0$.

Denote
\[
w:=(\sigma_1(A),\ldots,\sigma_k(A))^T,\mbox{  and  } a:=\sqrt{\sum_{i=k+1}^q \sigma_i^2(A)}.
\]
The reduced problem can be written as
\begin{equation*}
	\min_{s\in\mathcal{M}_k}  \sqrt{\|s-w\|_2^2+a^2} + \mu f_2\left(\begin{bmatrix} s \\ \mathbf{0}_{q-k} \end{bmatrix}\right),
\end{equation*}
where $\mathcal{M}_k:=\{s\in\R_+^k \mid s_1\ge \cdots \ge s_k\ge 0\}$.
When $F_2=\|\cdot\|_*$, the above model further reduces to
\begin{equation*}
	\min_{s\in\mathcal{M}_k} \sqrt{\|s-w\|_2^2+a^2} + \mu \|s\|_1.
\end{equation*}
If $s^*$ is an optimal solution of this problem, then
\[
X^*=U\Diag\left(\begin{bmatrix} s^* \\ \mathbf{0}_{q-k} \end{bmatrix}\right)V^T
\]
is an optimal solution to the original matrix problem \eqref{prob:low_rank_approx}. This recovers the results obtained in \cite[Theorem 4]{deng2025alternating}.

	\subsection{Weighted Regularized Spectral Approximation}

	We now turn our attention to the weighted regularized spectral approximation problem:
	\begin{equation} \label{prob2}
		\min_{X \in \R^{m \times n}} \ell(\|A - BXC\|_{\rm UI}) + \mu F(X) \quad \st \quad G(X)\leq c,
	\end{equation}
	where $A\in\R^{r\times s}$, $B\in\R^{r\times m}$, $C\in\R^{n\times s}$ are given matrices, and $\|\cdot\|_{\rm UI}$ denotes any unitarily invariant norm. In particular, when both $B$ and $C$ are identity matrices, \eqref{prob2} reduces to \eqref{prob:general} with $F_1(Z)=\ell(\|Z\|_{\rm UI})$. Denote $\bar q=\min\{r,s\}$. Let
	\begin{equation}\label{thinSVD-BC}
	B=U_B\Sigma_BV_B^T \quad \text{and} \quad C=U_C\Sigma_CV_C^T
	\end{equation}
	be thin SVDs of $B$ and $C$, respectively. Here, $r_b=\rank(B)$ and $r_c=\rank(C)$,
	\[
	\begin{aligned}
	&U_B\in\R^{r\times r_b},\quad V_B\in\R^{m\times r_b},\quad \Sigma_B\in\R^{r_b\times r_b},\\
	&U_C\in\R^{n\times r_c},\quad V_C\in\R^{s\times r_c},\quad \Sigma_C\in\R^{r_c\times r_c}.
	\end{aligned}
	\]
	Moreover, $U_B^\perp\in\R^{r\times(r-r_b)}$ and $V_C^\perp\in\R^{s\times(s-r_c)}$ denote orthogonal complements of $U_B$ and $V_C$, respectively.

	We make the following assumptions on  \eqref{prob2}.

	\begin{assumption} \label{ass:weighted_setup}
		For problem \eqref{prob2}, we assume that
		\begin{enumerate}[label=(\roman*),leftmargin=*]
			\item $F$ and $G$ are spectral functions on $\R^{m\times n}$, associated with absolutely symmetric functions $f,g:\R^q\to(-\infty,\infty]$, respectively;
			\item $\ell:\R_+\to(-\infty,\infty]$ is proper and non-decreasing;
			\item $f$ and $g$ are monotone under weak majorization on $\R^q_+$ in the sense of Definition~\ref{def:weak_majorization_monotonicity};
			\item  $A,B,C$ satisfy

\[
	(I-P_B)AP_C=0
	\quad \text{and} \quad
	P_BA(I-P_C)=0,
	\]
 where $P_B$ and $P_C$ are the orthogonal projectors
	onto the column space of $B$ and the row space of $C$,
	respectively;
			\item there exists a pair of thin SVDs of
$B$ and $C$ as in~\eqref{thinSVD-BC} such that $\widetilde A=U_B^TAV_C$ is a generalized diagonal matrix, i.e., $\widetilde A_{ij}=0$ for $i\ne j$.
		\end{enumerate}
	\end{assumption}

We make the following observations regarding
Assumption~\ref{ass:weighted_setup}. First, the thin SVDs of $B$
and $C$ need not be unique, particularly when they have repeated
positive singular values. Condition~(v) requires
that there exists a pair of thin SVDs satisfying the stated
diagonality condition. Whenever condition~(v) holds, we fix one
such pair throughout the subsequent analysis. All representations
below are understood with respect to these fixed SVDs.
Second, condition~(iv) is independent of the particular SVDs used.
Indeed, for any thin SVDs of $B$ and $C$, we have that
$ P_B=U_BU_B^T $ and $P_C=V_CV_C^T$. Since
$
I-P_B=U_B^\perp(U_B^\perp)^T$ and $
I-P_C=V_C^\perp(V_C^\perp)^T,
$
condition~(iv) is equivalent to
\begin{equation}\label{cond-orth}
(U_B^\perp)^TAV_C=0,
\qquad
U_B^TAV_C^\perp=0.
\end{equation}

	Assumption~\ref{ass:weighted_setup}(i) and (iii) cover a broad class of settings beyond purely norm-based formulations. Nevertheless, important cases such as rank regularization lie outside the present framework and are deferred to future work. The remaining conditions, (iv) and (v), impose structural requirements that are standard in the analysis of matrix approximation under weighted norms~\cite{yu2012rank}. The necessity of (iv) was established in \cite{yu2012rank}, and we will show via a concrete example that (v) is likewise essential in our setting.

	We require the following lemmas regarding unitarily invariant norms.
	The first result follows from the singular value inequalities in
	\cite[Problem II.5.5]{bhatia2013matrix},
	together with the Ky Fan dominance theorem \cite{fan1951maximum}.

	\begin{lemma} \label{lem:block_norm}
		For matrices of compatible dimensions and any unitarily invariant norm $\|\cdot\|_{\rm UI}$, it holds that
		\begin{equation*}
			\left\|\begin{pmatrix}
				A&B\\C&D
			\end{pmatrix}\right\|_{\rm UI}\geq \left\|\begin{pmatrix}
				A&0\\0&D
			\end{pmatrix}\right\|_{\rm UI}\geq\left\|\begin{pmatrix}
				A&0\\0&0
			\end{pmatrix}\right\|_{\rm UI}.
		\end{equation*}
	\end{lemma}

	As a generalization of Lemma~\ref{lem:block_norm}, we also have the following spectral weak majorization result.
		\begin{lemma}\label{lemma:maj}
			Let $A\in\R^{r_b\times r_c}$, and let $D\in\R^{(r-r_b)\times(s-r_c)}$. Define $A_d=\Diag(\diag(A))$ by setting all off-diagonal entries of $A$ to be zero. Then
			\[
			\sigma
			\begin{pmatrix}
				A_d  & 0 \\[2mm]
				0 & D
			\end{pmatrix}
			\prec_w
			\sigma
			\begin{pmatrix}
				A & 0 \\[2mm]
				0 & D
			\end{pmatrix}.
			\]
		\end{lemma}
		\begin{proof}
			We prove the lemma by showing that for $k=1,\dots, \bar q$,
			$$
			\left\|
			\begin{pmatrix}
				A_d  & 0 \\[2mm]
				0 & D
			\end{pmatrix}
			\right\|_{(k)} \le
			\left\|
			\begin{pmatrix}
				A  & 0 \\[2mm]
				0 & D
			\end{pmatrix}
			\right\|_{(k)},
			$$
			where $\|\cdot\|_{(k)}$ denotes the Ky Fan $k$-norm.
			Fix an arbitrary $k$. There exist $k_1$ and $k_2$ satisfying $k_1+k_2 \le k$ and
			$$
			\left\|
			\begin{pmatrix}
				A_d  & 0 \\[2mm]
				0 & D
			\end{pmatrix}
			\right\|_{(k)}
			=\| A_d\|_{(k_1)} + \|D \|_{(k_2)}.
			$$
With the convention \(\|\cdot\|_{(0)}=0\), if \(k_1=0\), the desired inequality is immediate. Hence, suppose \(k_1\ge1\).
			Since $\| \cdot \|_{(k_1)}$ is a unitarily invariant norm on $\R^{r_b\times r_c}$, we have the associated symmetric gauge function $\phi_{k_1}$ on $\R^{\min\{r_b,r_c\}}$ such that $\| \cdot \|_{(k_1)}= \phi_{k_1}(\sigma(\cdot))$.
			Additionally, by Theorem~\ref{thm:diag}(i), $|\diag(A_d)| = |\diag(A)| \prec_w \sigma(A)$. Then
			\begin{align*}
				& \| A_d\|_{(k_1)}
				=  \phi_{k_1}(\sigma(A_d))
				=  \phi_{k_1}(|\diag(A_d)|) \\
				\le \ &  \phi_{k_1}(\sigma(A))
				=  \| A\|_{(k_1)}.
			\end{align*}
			Thus, we have
			\begin{align*}
				\left\|
				\begin{pmatrix}
					A_d  & 0 \\[2mm]
					0 & D
				\end{pmatrix}
				\right\|_{(k)}
				\le \ & \| A\|_{(k_1)} + \|D \|_{(k_2)} \\
				\le \ &
				\left\|
				\begin{pmatrix}
					A  & 0 \\[2mm]
					0 & D
				\end{pmatrix}
				\right\|_{(k)}.
			\end{align*}
			The proof is completed.
	\end{proof}

	\begin{lemma}\label{lem:sv_contraction}
		Let $X\in\R^{m\times n}$, and let $P\in\R^{m\times m}$ and $Q\in\R^{n\times n}$ be orthogonal projectors. Then
		\[
		\sigma(PXQ)\prec_w \sigma(X).
		\]
	\end{lemma}
	\begin{proof}
		We prove the lemma by showing that for $k=1,\dots,q$,
		$$
		\|PXQ\|_{(k)} \leq \|X\|_{(k)},
		$$
		where $\|\cdot\|_{(k)}$ denotes the Ky Fan $k$-norm.
		Since $P$ and $Q$ are orthogonal projectors, there exist orthogonal matrices $U\in\R^{m\times m}$ and $V\in\R^{n\times n}$ such that
		\[
		U^TPU=\begin{pmatrix}I_{r_0}&0\\0&0\end{pmatrix},\qquad
		V^TQV=\begin{pmatrix}I_{s_0}&0\\0&0\end{pmatrix},
		\]
		where $r_0=\rank(P)$ and $s_0=\rank(Q)$. Define $\widetilde X=U^TXV$.
		Then
		\[
		U^T(PXQ)V=(U^TPU)\widetilde X(V^TQV)=\begin{pmatrix}\widetilde X_{11}&0\\0&0\end{pmatrix},
		\]
		where $\widetilde X_{11}$ is the leading $r_0\times s_0$ block of $\widetilde X$. Then by Lemma~\ref{lem:block_norm}, we have
		\begin{align*}
			& \|PXQ\|_{(k)} = \left\|\begin{pmatrix}\widetilde X_{11}&0\\0&0\end{pmatrix}\right\|_{(k)} \\
			\le \ & \| \widetilde X\|_{(k)}=\|  X\|_{(k)}.
		\end{align*}
		The proof is completed.
	\end{proof}

	\begin{proposition} \label{thm:prob2}
		Suppose Assumption~\ref{ass:weighted_setup} holds.
		If problem~\eqref{prob2} admits an optimal solution, then it admits an optimal solution of the form
		\[
		X^*= V_B \Sigma U_C^T,
		\]
		where $\Sigma$ is a generalized diagonal matrix.
	\end{proposition}

	\begin{proof}
		The proof is carried out in two steps. First, we show that if \eqref{prob2} admits an optimal solution $\bar{X}$, then $V_B V_B^T \bar{X} U_C U_C^T$ is also an optimal solution of \eqref{prob2}. Denote $\bar{Y} = V_B^T \bar{X} U_C \in \R^{r_b \times r_c}$. Second, we show that if $V_B \bar{Y} U_C^T$ is an optimal solution of \eqref{prob2}, then $V_B \bar{Y}_d U_C^T$ is also an optimal solution of \eqref{prob2}, where $\bar{Y}_d=\Diag(\diag(\bar{Y}))$ is obtained by setting all off-diagonal entries of $\bar{Y}$ to be zero.

		For notational simplicity, we assume $\ell(x)=x$ in the proof. The same proof applies directly to general non-decreasing $\ell$.
		Denote the orthogonal matrices $U=[\,U_B\  U_B^\perp\,]$ and $V=[\,V_C\  V_C^\perp\,]$. Denote $\widetilde A = U_B^T A V_C$. Then under Assumption~\ref{ass:weighted_setup}(iv), for any $X\in\R^{m\times n}$ we have
		\begin{align}\label{eq:block_form}
			&U^T(A-BXC)V\\
			=\ &
			\begin{pmatrix}
				U_B^TAV_C-\Sigma_BV_B^TXU_C\Sigma_C & U_B^TAV_C^\perp\\[2mm]
				(U_B^\perp)^TAV_C & (U_B^\perp)^TAV_C^\perp
			\end{pmatrix} \nonumber \\
			=\ & \begin{pmatrix}
				\widetilde A-\Sigma_BV_B^TXU_C\Sigma_C &  \\[2mm]
				& (U_B^\perp)^TAV_C^\perp
			\end{pmatrix}.\nonumber
		\end{align}

		Assume that $\bar{X}$ is an optimal solution of \eqref{prob2}. Now we prove the first step (showing that $V_B V_B^T \bar{X} U_C U_C^T$ is also optimal for \eqref{prob2}). By Lemma~\ref{lem:sv_contraction} and Assumption~\ref{ass:weighted_setup}, $F(V_B V_B^T \bar{X} U_C U_C^T) \leq F(\bar{X})$ and $G(V_B V_B^T \bar{X} U_C U_C^T) \leq G(\bar{X})$, implying that  $V_B V_B^T \bar{X} U_C U_C^T$ is feasible for \eqref{prob2}.
		By substituting $V_B V_B^T \bar{X} U_C U_C^T$ into $X$ of \eqref{eq:block_form}, we compute the objective value of \eqref{prob2} at $V_B V_B^T \bar{X} U_C U_C^T$
		\begin{align*}
			& \hspace*{-5mm}\left\|
			\left(
\begin{array}{@{}c@{\hspace{1pt}}c@{}}
				\widetilde A -\Sigma_BV_B^T (V_B V_B^T \bar{X} U_C U_C^T) U_C\Sigma_C &  \\[2mm]
				& (U_B^\perp)^TAV_C^\perp
			\end{array}
\right)
			\right\|_{\rm UI} \\
			&+  \mu F(V_B V_B^T \bar{X} U_C U_C^T) \\[2mm]
			&\le \left\|
			\begin{pmatrix}
				\widetilde A-\Sigma_BV_B^T \bar{X} U_C\Sigma_C &  \\[2mm]
				& (U_B^\perp)^TAV_C^\perp
			\end{pmatrix}
			\right\|_{\rm UI}   \\
			&+  \mu F(\bar{X}),
		\end{align*}
		where the last term  is exactly the objective value of \eqref{prob2} at the optimal solution $\bar{X}$. Thus $V_B V_B^T \bar{X} U_C U_C^T$ is optimal for \eqref{prob2}.

		In the second step, we show that $V_B \bar{Y}_d U_C^T$ is optimal for \eqref{prob2}, where $\bar{Y}_d=\Diag(\diag(\bar{Y}))$, $\bar{Y} = V_B^T \bar{X} U_C$, and $\bar{Y}_d,\bar{Y} \in \R^{r_b \times r_c}$.
		Since
		\begin{align*}
			[\,V_B\  V_B^\perp\,]^T V_B \bar{Y}_d U_C^T [\,U_C\  U_C^\perp\,]& =
			\begin{pmatrix}
				\bar{Y}_d& \\&0
			\end{pmatrix}, \\
			[\,V_B\  V_B^\perp\,]^T V_B \bar{Y} U_C^T [\,U_C\  U_C^\perp\,]& =
			\begin{pmatrix}
				\bar{Y}& \\&0
			\end{pmatrix},
		\end{align*}
		we have $\sigma(V_B \bar{Y}_d U_C^T) \prec_w \sigma(V_B \bar{Y} U_C^T)$ by Theorem~\ref{thm:diag}(i). It implies that $F(V_B \bar{Y}_d U_C^T) \le F(V_B \bar{Y} U_C^T)$ and $G(V_B \bar{Y}_d U_C^T) \le G(V_B \bar{Y} U_C^T)$, and thus $V_B \bar{Y}_d U_C^T$ is feasible for \eqref{prob2}.
		By substituting $V_B \bar{Y}_d U_C^T$ into $X$ of \eqref{eq:block_form} and using Lemma~\ref{lemma:maj}, we compute the objective value of \eqref{prob2} at $V_B \bar{Y}_d U_C^T$
		\begin{align*}
			& \left\|
			\begin{pmatrix}
				\widetilde A -\Sigma_B \bar{Y}_d \Sigma_C &  \\[2mm]
				& (U_B^\perp)^TAV_C^\perp
			\end{pmatrix}
			\right\|_{\rm UI} \nonumber \\
			&+  \mu F(V_B \bar{Y}_d U_C^T) \nonumber\\
			&\le \left\|
			\begin{pmatrix}
				\widetilde A-\Sigma_B \bar{Y} \Sigma_C &  \\[2mm]
				& (U_B^\perp)^TAV_C^\perp
			\end{pmatrix}
			\right\|_{\rm UI}   \nonumber \\
			&+  \mu F(V_B \bar{Y} U_C^T),
		\end{align*}
		where the last term  is exactly the objective value of \eqref{prob2} at the optimal solution $V_B \bar{Y} U_C^T$. Thus $V_B \bar{Y}_d U_C^T$ is optimal for \eqref{prob2}.
		The proof is completed.
	\end{proof}

	\begin{theorem} \label{thm:prob2_vector}
		Suppose Assumption~\ref{ass:weighted_setup} holds.
		Then problem~\eqref{prob2} admits an optimal solution if and only if the vector problem
		\begin{equation} \label{prob2:vector}
			\begin{aligned}
					\min_{x \in \R^p} \quad & \ell\big(\phi(\omega(x))\big) + \mu f\left(\begin{bmatrix}
						x \\
						\mathbf{0}_{q-p}
					\end{bmatrix}\right) \\
					\st \ \quad & g\left(\begin{bmatrix}
						x \\
						\mathbf{0}_{q-p}
					\end{bmatrix}\right) \le c
			\end{aligned}
		\end{equation}
		admits an optimal solution. Here
		\[
		\omega(x) =
		\sigma\!\left(
		\begin{bmatrix}
			\widetilde A -\Sigma_B \Diag(x) \Sigma_C &  \\[2mm]
			& (U_B^\perp)^TAV_C^\perp
		\end{bmatrix}
		\right),
		\]
		$p=\min(r_b,r_c)$, $q=\min(m,n)$, $\bar q=\min(r,s)$, and $\phi:\R^{\bar q}\to\R_+$ is the symmetric gauge function associated with the unitarily invariant norm $\|\cdot\|_{\rm UI}$.
		Furthermore, if $x^*$ is an optimal solution to~\eqref{prob2:vector}, then
		\[
		X^* = V_B \Diag(x^*) U_C^T
		\]
		is an optimal solution to~\eqref{prob2}.
		Conversely, if $X^*$ is an optimal solution to~\eqref{prob2}, then
		\[
		x^* = \diag(V_B^T X^* U_C)
		\]
		is an optimal solution to~\eqref{prob2:vector}.
	\end{theorem}

	\begin{proof}
		Let $x \in \R^p$ and define $X = V_B \Diag(x) U_C^T$.
		Then the singular values of $X$ are the absolute values of the entries of $x$, padded with zeros to length $q$.
		Hence
		\[
		F(X) = f\left(\begin{bmatrix}
			x \\
			\mathbf{0}_{q-p}
		\end{bmatrix}\right),
		\qquad
		G(X) = g\left(\begin{bmatrix}
			x \\
			\mathbf{0}_{q-p}
		\end{bmatrix}\right).
		\]
		Moreover, by \eqref{eq:block_form}, the first term in the objective of~\eqref{prob2} at $X$ is exactly the first term in the objective of~\eqref{prob2:vector} at $x$.
		Therefore, every feasible point of~\eqref{prob2:vector} yields a feasible point of~\eqref{prob2} with the same objective value.

		Conversely, let $X \in \R^{m \times n}$ be feasible for~\eqref{prob2}.
		By the proof of Proposition~\ref{thm:prob2}, the matrix
		\[
		\widehat X = V_B \Diag(\diag(V_B^T X U_C)) U_C^T
		\]
		is feasible for~\eqref{prob2} and has objective value no larger than that of $X$.
		Let $\widehat x = \diag(V_B^T X U_C)$.
		Then $\widehat x$ is feasible for~\eqref{prob2:vector}, and its objective value is exactly the objective value of~\eqref{prob2} at $\widehat X$.
		Hence the objective value of~\eqref{prob2:vector} at $\widehat x$ is no larger than that of~\eqref{prob2} at $X$.
		This proves the equivalence between~\eqref{prob2} and~\eqref{prob2:vector}.

		Assume that $x^*$ is an optimal solution of~\eqref{prob2:vector}.
		Define $X^* = V_B \Diag(x^*) U_C^T$.
		Let $X \in \R^{m \times n}$ be feasible for~\eqref{prob2}.
		Then, by the preceding argument, there exists a feasible point $\widehat x$ of~\eqref{prob2:vector} such that the objective value of~\eqref{prob2:vector} at $\widehat x$ is no larger than that of~\eqref{prob2} at $X$.
		We have
		\begin{align*}
			& \ell(\phi(\omega(x^*))) + \mu f\left(\begin{bmatrix}
				x^* \\
				\mathbf{0}_{q-p}
			\end{bmatrix}\right) \\
			\le \ & \ell(\phi(\omega(\widehat x))) + \mu f\left(\begin{bmatrix}
				\widehat x \\
				\mathbf{0}_{q-p}
			\end{bmatrix}\right) \\
			\le \ & \ell(\|A-BXC\|_{\rm UI}) + \mu F(X).
		\end{align*}
		Thus $X^*$ is an optimal solution to~\eqref{prob2}.

		Conversely, assume that $X^*$ is an optimal solution of~\eqref{prob2}.
		Let $x^* = \diag(V_B^T X^* U_C)$.
		By the proof of Proposition~\ref{thm:prob2}, the matrix $V_B \Diag(x^*) U_C^T$ is also optimal for~\eqref{prob2}.
		Let $x \in \R^p$ be feasible for~\eqref{prob2:vector}.
		Then $X = V_B \Diag(x) U_C^T$ is feasible for~\eqref{prob2}, and
		\begin{align*}
			& \ell(\phi(\omega(x^*))) + \mu f\left(\begin{bmatrix}
				x^* \\
				\mathbf{0}_{q-p}
			\end{bmatrix}\right) \\
			\le \ & \ell(\|A-BXC\|_{\rm UI}) + \mu F(X) \\
			= \ & \ell(\phi(\omega(x))) + \mu f\left(\begin{bmatrix}
				x \\
				\mathbf{0}_{q-p}
			\end{bmatrix}\right).
		\end{align*}
		Thus $x^*$ is an optimal solution to~\eqref{prob2:vector}.
	\end{proof}

	Our results in Theorem~\ref{thm:prob2_vector} can be applied to solve several concrete problems. A classical instance is the generalized regularized approximation problem in \cite{yu2012rank}
	\[
	\min_{X}\ \|A-BXC\|_{\rm UI}+\mu\|X\|_{{\rm UI}'},
	\]
	where $\|\cdot\|_{\rm UI}$ and $\|\cdot\|_{{\rm UI}'}$ are two unitarily invariant norms. Theorem~\ref{thm:prob2_vector} applies whenever
Assumption~\ref{ass:weighted_setup} holds.
	Another widely used instance is self-representation in subspace clustering \cite{liu2013lrr}
	\[
	\min_{X}\ \|X\|_{*}\quad \st\quad D=DX, \quad X\in\R^{N\times N},
	\]
	where $D\in\R^{d\times N}$ is the data matrix. This can be cast into \eqref{prob2} by taking $A=B=D$, $C=I$, $\mu=1$, omitting the explicit constraint $G$, and choosing $\ell(t)=0$ if $t=0$ and $\ell(t)=+\infty$ otherwise. Then $\ell(\|D-DX\|_{\rm UI})$ enforces the equality constraint $D=DX$. To verify the assumptions, let $D=U_D\Sigma_DV_D^T$ be a thin SVD and set
$Q=[V_D\  V_D^\perp]$. Using the SVD
$C=QIQ^T$, condition \eqref{cond-orth} holds because
$(U_D^\perp)^TDQ=0$ and $C=I$ is full rank, while Assumption~\ref{ass:weighted_setup}(v)
follows from $ U_D^TDQ=[\Sigma_D\ 0]$.
Hence Assumption~\ref{ass:weighted_setup} holds, and Theorem~\ref{thm:prob2_vector} applies.

The simultaneous diagonalization (SD) assumption, i.e., Assumption~\ref{ass:weighted_setup}(v), is essential for the diagonal reduction in Proposition~\ref{thm:prob2}.

	We show below that the conclusion of Proposition~\ref{thm:prob2} may fail when the SD assumption is violated, even if Assumption~\ref{ass:weighted_setup}(i)--(iv) hold.
	Consider
		\[
		\min_{X\in\R^{3\times3}} \|A-BXC\|_F^2+\mu\|X\|_F^2,
		\]
		where
		\begin{align*}
			A=\begin{pmatrix}
				0&1& \\
				 &0& \\
				 & &0
			\end{pmatrix},
			\qquad
			B=C=\begin{pmatrix}
				2&&\\
				&1&\\
				&&1
			\end{pmatrix}.
	\end{align*}
	We can see that for this problem, Assumption~\ref{ass:weighted_setup}(i)--(iii) hold. Moreover, since $B$ and $C$ are nonsingular, $P_B=P_C=I$, and hence Assumption~\ref{ass:weighted_setup}(iv) holds.

Next, we show that the SD assumption fails for every pair of thin SVDs of $B$ and $C$. Since the singular value $2$ is simple and the singular value $1$ has multiplicity two, any such SVDs, with the singular values arranged in nonincreasing order, satisfy
	\[
	U_B=V_B=\begin{pmatrix}\varepsilon_B&0\\0&R_B\end{pmatrix},
	\quad
	U_C=V_C=\begin{pmatrix}\varepsilon_C&0\\0&R_C\end{pmatrix},
	\]
	where $\varepsilon_B,\varepsilon_C\in\{-1,1\}$ and $R_B,R_C\in\R^{2\times2}$ are orthogonal matrices. A direct calculation gives
\[
	\widetilde A=U_B^TAV_C=
	\begin{pmatrix}
		0 & \varepsilon_B (R_C)_{11}& \varepsilon_B (R_C)_{12} \\
0 & 0 & 0 \\
0 & 0 & 0
	\end{pmatrix},
	\]
 which cannot be a generalized diagonal matrix since $R_C$ is orthogonal. Thus, no pair of thin SVDs satisfies the SD assumption.
	Let
	\[
	L(X)=\|A-BXC\|_F^2+\mu\|X\|_F^2.
	\]
	The first-order optimality condition is
	\[
	\nabla_XL(X)=-2B^T(A-BXC)C^T+2\mu X=0,
	\]
	or equivalently,
	\[
	B^TBXC C^T+\mu X=B^TAC^T.
	\]
	Since $B$ and $C$ are diagonal, this system decouples elementwise as
	\[
	B_{ii}^2C_{jj}^2X_{ij}+\mu X_{ij}=B_{ii}A_{ij}C_{jj}, \qquad i,j=1,2,3.
	\]
	Since $B$ and $C$ are nonsingular and $\mu\ge0$, $L$ is strictly
convex. Hence its unique minimizer is given by
	\[
	X^*_{ij}
	=\frac{B_{ii}A_{ij}C_{jj}}{B_{ii}^2C_{jj}^2+\mu }
	,\quad X^* =\frac{1}{4+\mu}
	\begin{pmatrix}
		0&2&0\\
		0&0&0\\
		0&0&0
	\end{pmatrix}.
	\]
For any diagonal matrix
$\Sigma=\Diag(\Sigma_{11},\Sigma_{22},\Sigma_{33})$, we have
\[
V_B\Sigma U_C^T
=
\begin{pmatrix}
\varepsilon_B\varepsilon_C\Sigma_{11}&0\\
0&R_B\Diag(\Sigma_{22},\Sigma_{33})R_C^T
\end{pmatrix}.
\]
Hence the optimal solution $X^*$ cannot be written in the form $V_B\Sigma U_C^T$ with $\Sigma$ diagonal. This shows that the conclusion of Proposition~\ref{thm:prob2} may fail without the SD assumption.

	\section{Conclusion}
In this paper, we investigated exact low-dimensional reformulations for regularized spectral approximation problems.
The unweighted model reduces to a singular-value formulation under weak-majorization monotonicity, while the weighted case admits such a reduction under compatibility and simultaneous diagonalization.

Our results characterize when spectral matrix approximation problems can be treated via lower-dimensional representations, and demonstrate that these structural conditions are essential, as the reduction may fail without them.
Future directions include relaxing these assumptions, extending the framework to more general settings, and designing efficient algorithms for the reduced formulations.

\end{document}